\documentclass[12pt]{article}
\usepackage{a4wide,comment,todonotes,mathrsfs,amsmath,amssymb,mathpazo}
\usepackage{tikz}
\usetikzlibrary{calc}

\usepackage[T1]{fontenc}
\usepackage{esint}

\newcommand{\Eqref}[1]{Eq.\,(\ref{#1})}

\usepackage{multirow}
\usepackage{tabulary}

\newtheorem{thm}{Theorem}

\newtheorem{exl}[thm]{Example}

\newenvironment{proof}{\vspace{-0.1cm}\noindent\textbf{Proof:}}{$\blacksquare$\\}

\newcommand{\ep}{\epsilon}

\newcommand{\NN}{\mathbb{N}}
\newcommand{\RR}{\mathbb{R}}
\newcommand{\PP}{\mathbb{P}}
\newcommand{\EE}{\mathbb{E}}

\newcommand{\MDP}{{\small MDP}}
\newcommand{\MDPs}{{\small MDP}s}

\begin{document}
\tikzstyle{state}=[circle,thick,draw=black!80]
\tikzstyle{term}=[rectangle,thick,draw=black!80]
\tikzstyle{andsoon}=[rectangle,thick,draw=white!80,fill=white!20]
\tikzstyle{andsoonfill}=[rectangle,rounded corners, text width=8em, text centered,thick,draw=black!80,fill=black!10]
\tikzstyle{nameex}=[rectangle,thick]

\title{Optimal strategies in Markov decision processes with finitely additive evaluations\footnote{We would like to thank Fabio Maccheroni for a very useful discussion.}}
\author{J\'{a}nos Flesch\footnote{Department of Quantitative Economics, Maastricht University, P.O.Box 616, 6200 MD, The Netherlands. Email: jflesch@maastrichtuniversity.nl}\ \and Arkadi Predtetchinski\footnote{Department of Microeconomics and Public Economics, Maastricht University, P.O.Box 616, 6200 MD, The Netherlands. Email: a.predtetchinski@maastrichtuniversity.nl}\ \and William Sudderth\footnote{School of Statistics, University of Minnesota, Minneapolis, MN 55455, United States. Email: bill@stat.umn.edu}\ \and Xavier Venel\footnote{Dipartimento di AI, Data and Decision Sciences, Luiss University, Viale Romania 32, 00197 Rome, Italy. Email: xvenel@luiss.it}}
\maketitle

\begin{abstract}
We study infinite-horizon Markov decision processes (\MDPs) where the decision maker evaluates each of her strategies by aggregating the infinite stream of expected stage-rewards. The crucial feature of our approach is that the aggregation is performed by means of a given diffuse charge (a diffuse finitely additive probability measure) on the set of stages. The results of Neyman [2023] imply that in this setting, in every \MDP\ with finite state and action spaces, the decision maker has a pure optimal strategy as long as the diffuse charge satisfies the time value of money principle. His result raises the question of existence of an optimal strategy without additional assumptions on the aggregation charge. We answer this question in the negative with a counterexample. With a delicately constructed aggregation charge, the \MDP\ has no optimal strategy at all, neither pure nor randomized.
\end{abstract}

\section{Introduction}
Markov decision processes (\MDP s) are one of the most frequently used models for describing and examining dynamic decision making. \MDP s are covered in a very extensive literature and they have been applied in various contexts (e.g.~Puterman [1990]). 

We study \MDP s with finite state and action spaces where the set of stages is $\NN=\{1,2,\ldots\}$. At each stage, the play of the \MDP\  is in one of the states, where the decision maker chooses one of the actions. Depending on this action, the decision maker receives an instantaneous reward, and the next state is chosen according to a probability distribution, where the play continues at the next stage.  

Each strategy of the decision maker defines an expected reward for every stage, and thus determines an infinite stream of expected rewards. We assume that the decision maker evaluates each strategy by aggregating the corresponding infinite stream of expected rewards into a real number, called payoff, by means of a diffuse charge, that is, a finitely additive probability measure on the set $\NN$ of stages such that each stage has weight zero. We will call this diffuse charge an \emph{aggregation charge}. 

A special case of this setup is the well-known class of \MDP s with the long-term average reward. This case is attained if the aggregation charge assigns the limiting frequency to each subset of the set of stages, whenever this limiting frequency exists (cf.~Example~\ref{ex-averagemdp}). 

Another important special case is the class of \MDP s where the aggregation charge satisfies the time value of money principle. This class is considered in Neyman [2023], although in his paper aggregation charges are formulated in terms of linear functionals (cf.~Section~\ref{sec-prel}). Neyman [2023] showed that, in every \MDP, the decision maker has a pure stationary strategy that is optimal for all aggregation charges that satisfy the time value of money principle (cf.~Theorem~\ref{theorem-Neyman}). 

When aggregation charges do not satisfy the time value of money principle, it can occur that the decision maker has a pure optimal strategy, yet none of the stationary strategies achieves optimality (cf.~Example \ref{ex-nostat}). Therefore, the question naturally arises whether, in each \MDP\ with each aggregation charge, the decision maker possesses an optimal strategy. In this paper, we answer this question in the negative (cf.~Theorem~\ref{theorem-counter}) by means of a counterexample. We present an \MDP\ with deterministic transitions and a delicately constructed diffuse charge such that the decision maker has no optimal strategy at all, neither pure nor randomized.

The paper is structured as follows. In Section \ref{sec-prel}, we 
discuss some preliminaries. In Section \ref{sec-mod}, we define the model. In Section \ref{sec-opt}, we recall the result of Neyman [2023], and present our main result, Theorem \ref{theorem-counter}. Finally, in Section \ref{sec-concluding}, we make a few concluding remarks.

\section{Preliminaries}\label{sec-prel}
In the entire paper, we assume the axioms of ZFC. Let $\NN=\{1,2,\ldots\}$ denote the set of natural numbers.\medskip

\noindent\textbf{Charges.} A \emph{charge}, also called finitely additive probability measure, on $\NN$ is a mapping $\mu:2^\NN\to[0,1]$ such that $\mu(\emptyset)=0$, $\mu(\NN)=1$, and $\mu$ satisfies finite additivity: for all disjoint $M_1\subseteq \NN$ and $M_2\subseteq \NN$ we have $\mu(M_1\cup M_2)=\mu(M_1)+\mu(M_2)$. Note that a charge $\mu$ is defined  on the entire power set $2^\NN$, which is without loss of generality, as every charge defined on an algebra of $\NN$ has an extension (generally not unique) to the entire power set. We denote by $\Delta^f$ the set of charges on $\NN$. For a detailed discussion of charges, we refer to Rao and Rao [1983]. 

A charge $\mu\in \Delta^f$ is called \emph{diffuse} if\footnote{With a slight abuse of notation, we denote by $\mu(n)$ the $\mu$-measure of the singleton $\{n\}$.} $\mu(n)=0$ for every $n\in\NN$. We denote by $\Delta^{d}$ the set of diffuse charges on $\NN$. The set $\Delta^{d}$ is a strict subset of $\Delta^f$. 

A diffuse charge $\mu\in\Delta^d$ is called a \emph{frequency charge} if
\begin{equation}\label{def-freq}
\mu(M)\,=\,\lim_{n\to\infty}\,\frac{|\{1,\ldots,n\}\cap M|}{n}
\end{equation}
for each set $M\subseteq \NN$ for which the limit in the righthandside exists. Moreover, a diffuse charge $\mu\in\Delta^d$ is called \emph{translation invariant} if $\mu(M)=\mu(M+1)$ for all $M\subseteq \NN$, where $M+1 = \{n+1: n \in M\}$. Frequency charges exist (cf.~Theorem 3 in Kadane and O'Hagan [1995]), and they form a strict subset of the collection of translation invariant charges (cf.~Proposition 3.3 and the non-trivial Example 4.8 in Schirokauer and Kadane [2007]).\medskip

\noindent\textbf{Integration with respect to a charge.} Consider a charge $\mu\in\Delta^f$. A function $s \colon \NN \rightarrow \RR$ is called a simple function if, for some $m\in\NN$, there are $c_1, \ldots, c_m \in \RR$ and a partition $\{M_1, \ldots, M_m\}$ of $\NN$ such that $s  = \sum_{i=1}^m c_i \mathbb{I}_{M_i}$, where $\mathbb{I}_{M_i}$ is the characteristic function of the set $M_i$. With respect to $\mu$, the integral of $s$ is defined by $\int_{t\in\NN} s(t) \mu(dt) = \sum_{i=1}^m c_i \cdot \mu(M_i)$.
\medskip

For a bounded function $f \colon \NN \rightarrow \RR$, the integral $\int_{t\in\NN} f(t) \mu(dt)$ is defined as the supremum of all real numbers $\int_{t\in\NN} s(t) \mu(dt)$, where $s$ is a simple function and $s \le f$. This is equivalent to taking the infimum of those real numbers, where $s$ is a simple function and $s \ge f$. Since $f$ is bounded, the integral is finite.\medskip

\noindent\textbf{Charges satisfying the time value of money principle.} We will now lay the connection with the setup in Neyman [2023]. As we explain below, there is a natural one-to-one correspondence between patient valuations as defined in Neyman [2023] and the class\footnote{This class of charges is a strict subset of the class of frequency charges. Indeed, the inclusion follows by \Eqref{eq-betwN}, and the strictness of the inclusion is demonstrated by Example 4.5 in Schirokauer and Kadane [2007].} of diffuse charges $\mu\in\Delta^d$ satisfying  
\begin{equation}\label{eq-betwN}
\liminf_{n\to\infty}\,\frac{|\{1,\ldots,n\}\cap M|}{n}\,\leq\,\mu(M)\,\leq\,\limsup_{n\to\infty}\,\frac{|\{1,\ldots,n\}\cap M|}{n}
\end{equation}
for each $M\subseteq \NN$. 

Neyman [2023, Definition 3] defines the notion of a \emph{valuation} as a function $w:\ell^\infty\to\RR$, where $\ell^\infty$ is the set of all bounded infinite sequences in $\RR$, such that: (i) $w$ is normalized: $w(1,1,\ldots)=1$, (ii) $w$ is additive: $w(g+g')=w(g)+w(g')$ for all $g,g'\in\ell^\infty$, and (iii) $w$ obeys the time value of money principle: for all $g=(g_1,g_2,\ldots)\in\ell^\infty$ and $g'=(g'_1,g'_2,\ldots)\in\ell^\infty$, if the partial sums satisfy $g_1+\cdots+g_n\geq g'_1+\cdots+g'_n$ for each $n\in\NN$, then $w(g)\geq w(g')$. 

A valuation $w:\ell^\infty\to\RR$ is called \emph{patient} in Neyman [2023, Definition 6] if $w(g)=1$ for every sequence $g\in\ell^\infty$ such that, for some $k\in\NN$, we have $g_n=0$ for all $n\leq k$ and $g_n=1$ for all $n>k$. By the characterization in Theorem 2 in Neyman [2023], a function $w:\ell^\infty\to\RR$ is a patient valuation if and only if $w$ is linear and satisfying
\begin{equation}\label{eq-n2}
\liminf_{n\to\infty}\,\frac{g_1+\cdots+g_n}{n}\,\leq\,w(g)\,\leq\,\limsup_{n\to\infty}\,\frac{g_1+\cdots+g_n}{n}
\end{equation}
for every $g=(g_1,g_2,\ldots)\in \ell^\infty$. It follows that every patient valuation is continuous with respect to the $||\cdot||_\infty$-norm. 

For every patient valuation $w:\ell^\infty\to\RR$, since $w$ is linear and continuous with respect to the $||\cdot||_\infty$-norm, there is a unique diffuse charge $\mu\in\Delta^d$ such that \begin{equation}\label{eq-corrcharge}
w(g)\,=\,\int_{n\in\NN} g_n\; \mu(dn)
\end{equation}
for all $g\in\ell^\infty$; see for instance\footnote{This theorem applies, using the notation in Rao and Rao [1983], as follows: $\Omega:=\NN$, $\mathscr{F}:=2^\NN$, and $T:=w$. Then, in view of Definition 4.7.1 in Rao and Rao [1983], the space $\mathscr{C}(\Omega,\mathscr{F})$ equals $\ell^\infty$, because $\mathscr{F}$ contains all subsets of $\Omega$. Hence, the equation in Theorem 4.7.4 in Rao and Rao [1983] implies that \Eqref{eq-corrcharge} holds for a unique bounded  finitely additive signed-measure $\mu$. As $T(g)=w(g)\geq 0$ if $g_n\geq 0$ for all $n\in\NN$, the last claim of Theorem 4.7.4 implies that $\mu$ only takes nonnegative values. As $T(1,1,\ldots)=w(1,1,\ldots)=1$ by property (i) of a valuation, we obtain $\mu(\NN)=1$. Finally, $\mu$ is diffuse because $T=w$ is patient.} Theorem 4.7.4 in Rao and Rao [1983]. Note that this charge $\mu$ satisfies \Eqref{eq-betwN} due to \Eqref{eq-n2}. 

In conclusion, each patient valuation $w$ can be written as an integral in \Eqref{eq-corrcharge} with respect to a diffuse charge $\mu\in\Delta^d$ satisfying \Eqref{eq-betwN}. The converse is also true: if $\mu\in\Delta^d$ satisfies \Eqref{eq-betwN}, then $w$ given by \Eqref{eq-corrcharge} is a patient valuation.\medskip

\noindent\textbf{The topology of pointwise convergence on the set of charges.} We will consider a specific topology on the set $\Delta^f$ of charges: the \emph{topology of pointwise convergence}. This is the smallest topology\footnote{This topology is the same as the topology that $\Delta^f$ inherits as a closed subspace of the compact product space $[0,1]^{(2^\NN)}$.} on $\Delta^f$ under which, for each $M\subseteq \NN$, the mapping $\phi_M:\Delta^f\to[0,1]$ defined by $\phi_M(\mu)=\mu(M)$ is continuous. This topological space is compact and Hausdorff, but not metrizable. 

A net in $\Delta^f$ is a collection $\{\mu_\alpha\}_{\alpha\in \Omega}$, where $(\Omega,\succeq)$ is a directed set and $\mu_\alpha\in \Delta^f$ for every $\alpha\in\Omega$. We will be mostly interested in sequences in $\Delta^f$, a special case of nets in $\Delta^f$ where the directed set is $(\NN,\geq)$.

Consider a sequence $\{\mu_n\}_{n\in\NN}$ in $\Delta^f$. A charge $\mu\in \Delta^f$ is an \emph{accumulation point} of $\{\mu_n\}_{n\in\NN}$ if for every finite collection $\{M_1,\ldots,M_k\}$ of subsets of $\NN$ and for every $\ep>0$ there exists a sequence $m_1<m_2<\cdots$ such that 
\begin{equation*}
\forall n\in\NN,\ \forall i=1,\ldots,k:\hspace{0.5cm}|\mu(M_i)-\mu_{m_n}(M_i)|\;\leq\;\ep.
\end{equation*}
Equivalently, if the sequence $\{\mu_n\}_{n\in\NN}$ has a subnet converging to $\mu$. As $\Delta^f$ is compact, each sequence $\{\mu_n\}_{n\in\NN}$ in $\Delta^f$ has at least one accumulation point. \medskip

%\begin{exl}\rm 
%Positive MDPs. Take the set $X$ of the theorem to be the space $Z$ of plays for the MDP and let
%$f(z,n) = \sum_{t=1}^n r_t(z)$. Also let the measure $\nu$ of the theorem be 
%the distribution on $Z$ determined by the strategy $\sigma$. Then I think 
%that the theorem applies to show that expressions 1 and 2 of section 8 are 
%equal, and both are equal to $\int_X f(x,\infty)\,d\nu(x) 
%=\mathbb{E}_{\sigma}[\sum_{t=1}^{\infty} r_t]$. In particular, this quantity 
%does not depend on the diffuse measure $\mu$.
%\end{exl}
%
%If the measure $\mu$ is countably additive, the equality of expressions 1 and 2 of section 8 follows from the usual Fubini theorem. The equality for a 
%general finitely additive $\mu$ follows because it can be written as a convex 
%combination of a diffuse measure and a countably additive measure.

\section{The model}\label{sec-mod}

\subsection{Markov decision processes}\label{secsub-mdpmodel}

\textbf{The definition of a Markov decision process (\MDP).} An \MDP\ is given by:
\begin{itemize}
    \item A nonempty finite state space $S$, with a given state $s^*\in S$ as the initial state.
    \item A nonempty finite action space $A(s)$ for each state $s\in S$. 
    \item A bounded reward $r(s,a)\in\RR$ for each state $s\in S$ and action $a\in A(s)$.
    \item A probability measure $p(s,a)=(p(z|s,a))_{z\in S}$ on the state space $S$, for each state $s\in S$ and action $a\in A(s)$.
\end{itemize}

An \MDP\ is played on the infinite horizon, at stages in $\mathbb{N}$ as follows: At stage 1, in the given initial state $s_1 = s^*$, the decision maker chooses an action $a_1\in A(s_1)$. Then, she receives reward $r(s_1,a_1)$, and subsequently transition occurs to a state $s_2\in S$ according to the probability measure $p(s_1,a_1)$. At stage 2, in state $s_2$, the decision maker again chooses an action $a_2\in A(s_2)$. Then, she receives reward $r(s_2,a_2)$, and subsequently transition occurs to a state $s_3\in S$ according to the probability measure $p(s_2,a_2)$, and so on.

The transitions in the \MDP\ are said to be \emph{deterministic} if for each state $s\in S$ and each action $a\in A(s)$ there is a state $s'\in S$ such that $p(s'|s,a)=1$.\medskip

\noindent\textbf{Strategies.} A \emph{history} at stage $t\in\mathbb{N}$ is a sequence $(s_1=s^*,a_1,\ldots,s_{t-1},a_{t-1},s_t)$ such that $s_i \in S$ for all $i=1,\ldots,t$, and $a_i\in A(s_i)$ for all $i=1,\ldots,t-1$. 

A (behavioral) \emph{strategy} $\sigma$ is a map that, to each history $h=(s_1,a_1,\ldots,s_{t-1},a_{t-1},s_t)$ at any stage $t\in\NN$, assigns a probability measure $\sigma(h)$ on the action space $A(s_t)$. The interpretation is that $\sigma$ recommends to choose an action randomly according to $\sigma(h)$ if history $h$ arises. The set of all strategies is denoted by $\Sigma$. 

A strategy $\sigma$ is called \emph{pure} if $\sigma(h)$ assigns probability one to an action for any history~$h$; otherwise the strategy is called \emph{randomized}. A strategy $\sigma$ is called \emph{stationary} if $\sigma(h)$ only depends on the the history $h=(s_1,a_1,\ldots,s_{t-1},a_{t-1},s_t)$ through its current state $s_t$; that is, a stationary strategy can be seen as a collection $(\sigma_s)_{s\in S}$ where $\sigma_s$ is a probability measure on $A(s)$. 

\subsection{Aggregation charge}\label{mu-average}

We assume in this paper that the rewards that the decision maker receives during the \MDP\ are aggregated by a given diffuse charge $\mu\in\Delta^d$, which we also call the \emph{aggregation charge}. The \MDP\ supplemented with an aggregation charge $\mu\in\Delta^d$ is denoted by \MDP($\mu$).\medskip

\noindent\textbf{The payoff induced by a strategy.} Consider an \MDP$(\mu)$. The following evaluation was proposed by Neyman [2023], although in a somewhat different form (cf.~Section \ref{sec-prel}). 

It is natural to evaluate every strategy $\sigma$ by the \emph{payoff} defined as
\begin{equation}\label{eqpay}
u_\mu(\sigma)\;:=\;\int_{t\in\NN}{\EE_{\sigma}[r_t]}\;\mu(dt),
\end{equation}
where $r_t$ denotes the random variable for the reward at stage $t\in\NN$, and $\EE_\sigma$ denotes the expectation operator under $\sigma$. That is, the payoff $u_\mu(\sigma)$ is the $\mu$-weighted aggregation, or $\mu$-weighted average, of the expected rewards at all stages. Note that $u_\mu(\sigma)$ is well-defined for any strategy $\sigma$ and any aggregation charge $\mu$, because the sequence $\{\EE_{\sigma}[r_t]\}_{t\in\NN}$ is bounded for any strategy $\sigma$, and hence this sequence is integrable with respect to any charge $\mu$ (cf.~Section \ref{sec-prel}).

\begin{exl}\label{ex-averagemdp}\rm
If $\mu\in\Delta^d$ is a frequency charge, the payoff $u_{\mu}$ defined in \Eqref{eqpay} leads to an \MDP\ with the long-term average reward. For instance, if a strategy $\sigma$ induces an expected reward of 1 in each odd stage and an expected reward of 0 in each even stage, then in view of \Eqref{def-freq}, we have  $u_\mu(\sigma)=\mu(\{1,3,5,\ldots\})=\frac{1}{2}$. $\blacklozenge$
\end{exl}

\noindent\textbf{The value and optimal strategies.} With respect to the payoff in \Eqref{eqpay}, the \emph{value} of the \MDP($\mu$) is defined as \[v_\mu\,:=\,\sup_{\sigma\in\Sigma}\,u_\mu(\sigma).\]
A strategy $\sigma\in\Sigma$ is called \emph{optimal} if $u_\mu(\sigma)= v_\mu$.

\section{Non-existence of optimal strategies}\label{sec-opt}
Our starting point is the following result, which follows from Theorem 6 in Neyman [2023] when applied to diffuse charges. 

\begin{thm}[Neyman {[}2023, Theorem 6{]}]\label{theorem-Neyman}
Consider an \MDP\ with finite state and action spaces. Then, the decision maker has\footnote{One can choose any Blackwell optimal strategy for $\sigma$, i.e., any pure stationary strategy that is optimal for all sufficiently small discount rates. (In the proof of Theorem 6 in Neyman [2023] and the notation there, $\omega_t=0$ for all stages $t$, and hence for any $\ep>0$, one can choose $t_\ep=1$, which implies that the constructed strategy $\tau$ coincides with a Blackwell optimal strategy.)} a pure stationary strategy $\sigma$ such that $\sigma$ is optimal in the \MDP$(\mu)$ for all aggregation charges $\mu\in\Delta^d$ satisfying \Eqref{eq-betwN}.
\end{thm}

Note that the strategy $\sigma$ in Theorem \ref{theorem-Neyman} is robust: $\sigma$ does not depend on the choice of the aggregation charge $\mu$, as long as $\mu$ satisfies \Eqref{eq-betwN}. 

When allowing aggregation charges $\mu$ that do not satisfy \Eqref{eq-betwN}, it is not difficult to find \MDPs\ such that the decision maker has a pure optimal strategy in the \MDP$(\mu)$, but none of the stationary strategies is optimal; cf.~Example \ref{ex-nostat}. 

The question naturally arises whether, for each \MDP\ and each aggregation charge $\mu$, the decision maker has an optimal strategy in the \MDP$(\mu)$. This brings us to the main contribution of the paper.\medskip 

\noindent\textbf{The even--or--odd \MDP$(\mu)$ :} First we define the \MDP, then the aggregation charge $\mu\in\Delta^d$.

Consider the \MDP\ such that (i) the state space is $S=\{1,2,3\}$, (ii) there are two actions $T$ and $B$ in state 1 and only one action in states 2 and 3, (iii) in state 1 the reward is 1 for action $T$ and 0 for action $B$, in state 2 the reward is 0, and in state 3 the reward is 1, (iv) action $T$ leads to state 2, action $B$ leads to state 3, and from states 2 and 3 the transition is to state 1. The \MDP\ is represented in the figure below:
\vspace{0.4cm}
\\

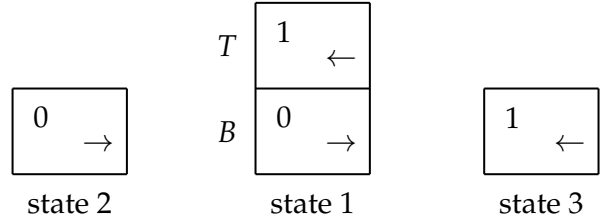
\scalebox{0.95}{
\hspace{3cm}\begin{tikzpicture}[x=0.4cm,y=0.4cm]
\draw[thick] (-8.5,0) -- (-4.5,0);
\draw[thick] (-8.5,3) -- (-4.5,3);
\draw[thick] (-8.5,0) -- (-8.5,3);
\draw[thick] (-4.5,0) -- (-4.5,3);
\node at (-7.5,2) {$0$};
\node at (-5.5,1) {$\rightarrow$};
\node at (-6.5,-1) {state 2};

\node at (-1,4.5) {$T$};
\node at (-1,1.5) {$B$};
\draw[thick] (0,0) -- (4,0);
\draw[thick] (0,3) -- (4,3);
\draw[thick] (0,6) -- (4,6);
\draw[thick] (0,0) -- (0,6);
\draw[thick] (4,0) -- (4,6);
\node at (1,5) {1};
\node at (1,2) {0};
\node at (3,4) {$\leftarrow$};
\node at (3,1) {$\rightarrow$};
\node at (2,-1) {state 1};

\draw[thick] (8,0) -- (12,0);
\draw[thick] (8,3) -- (12,3);
\draw[thick] (8,0) -- (8,3);
\draw[thick] (12,0) -- (12,3);
\node at (9,2) {$1$};
\node at (11,1) {$\leftarrow$};
\node at (10,-1) {state 3};
\end{tikzpicture}
}
\vspace{0.2cm}

Suppose that the initial state is state 1. Essentially, the stages can be grouped in blocks of two: stages 1 and 2, stages 3 and 4, stages 5 and 6, etc. At every odd stage the decision-maker chooses the payoff for the current stage and the next one, where her options are only two: getting 1 now and 0 at the next stage, or vice versa, 0 now and 1 at the next stage. Note that the transitions are deterministic. 

Now we define the charge $\mu$. The construction of $\mu$ is quite intricate and the details are critical to our proof. 

Let $\phi\in\Delta^d$ be a frequency charge (cf.~Section \ref{sec-prel}); note that $\phi$ is then translation invariant. Let $E_{0} = \{1,3,5,\ldots\}$ and for each $n \in \NN$ let $E_{n} = \{k\cdot 2^n\colon k\in\NN\}$.  So, for example, $E_1=\{2,4,6,8,10,12,14,16,\ldots\}$, $E_2=\{4,8,12,16,\ldots\}$, $E_3=\{8,16,\ldots\}$; in general, for $n\geq 1$, the set $E_{n+1}$ consists of every second element of the set $E_n$. Thus, the sequence $E_1,E_2,E_3,\ldots$ is decreasing with respect to set-inclusion, and $E_1\cap E_2\cap E_3\cap\cdots=\emptyset$.

Define the diffuse charge $\mu_{0}\in\Delta^d$ by $\mu_0(W)=2\cdot \phi(W\cap E_0)$ for each set $W \subseteq \NN$. This is essentially the charge $\phi$ applied to the set of odd numbers; in particular, $\mu_{0}(E_{0}) = 2\cdot \phi(E_0)=1$. Similarly, for each $n \in \NN$ define the diffuse charge $\mu_{n}\in\Delta^d$ by $\mu_n(W)=2^n\cdot \phi(W\cap E_n)$ for each set $W \subseteq \NN$. The charge $\mu_{n}$ is thus the charge $\phi$ applied to the set $E_{n}$; in particular, $\mu_{n}(E_{n}) = 2^n\cdot \phi(E_n)=1$. 

We endow the set $\Delta^f$ of charges on $\NN$ with the topology of pointwise convergence (cf.~Section \ref{sec-prel}). Since $\Delta^f$ is a compact space, the sequence $\{\mu_{n}\}_{n=1}^\infty$, as a net, has an accumulation\footnote{One such accumulation point is the charge $\mu_*$ such that $\mu_*(W)=\int_{n\in\NN}\mu_n(W)\,\phi(dn)$ for all $W\subseteq \NN$.} point $\mu_{*}\in\Delta^f$. Because each element of this sequence is a diffuse charge, we have $\mu_*\in\Delta^d$.

It holds that $\mu_{*}(E_{n}) = 1$ for each $n \in \NN$. Indeed, whenever $m\geq n$, we have $E_n\cap E_m=E_m$ and hence $\mu_m(E_n)=2^m\cdot \phi(E_n\cap E_m)=2^m\cdot \phi(E_m)=1$. 

Now let the aggregation charge be $\mu = \tfrac{1}{2}\mu_{0}+\tfrac{1}{2}\mu_{*}$. Clearly, $\mu\in\Delta^d$, because $\mu_0\in\Delta^d$ and $\mu_*\in\Delta^d$.\medskip

\noindent\textbf{Intuition:} The idea of this \MDP\ is as follows. As mentioned above, in each odd stage, the decision maker has to choose between action $T$ (giving reward 1 now and 0 at the next stage) and action $B$ (giving reward 0 now and 1 at the next stage). 

There is a tension between playing well for the charge $\mu_{0}$ and for the charge $\mu_{*}$. The charge $\mu_0$ is concentrated on the set of odd stages, while $\mu_{*}$ on the set of even stages. To play well for $\mu_0$, the decision maker needs to choose action $T$ as frequently as possible. In contrast, to play well for the charge $\mu_*$, the decision maker only needs to choose action $B$ with a positive frequency. Indeed, any of the sets $E_n$, for $n\geq 1$, has frequency $\frac{1}{2^n}$ but $\mu_*$-measure 1.

To perform optimally for $\mu$, the decision maker must perform well with respect to its both components, $\mu_{0}$ and $\mu_{*}$. For any $\ep\in(0,1)$, playing action $T$ with frequency $1-\ep$ gives a payoff of $1-\ep$ for $\mu_{0}$, a payoff of 1 for $\mu_{*}$, and hence a payoff of $1-\frac{1}{2}\ep$ for $\mu$. Yet, playing action $T$ with frequency 1 ignores $\mu_{*}$ and gives a payoff of only $\frac{1}{2}$ for $\mu$. This illustrates that obtaining the payoff of 1 is impossible.\medskip 

\begin{thm}\label{theorem-counter}
In the even--or--odd \MDP$(\mu)$ the decision maker has no optimal strategy, neither pure nor randomized.
\end{thm}

\begin{proof} The statement follows by the two claims below.\smallskip\\
\noindent\textsc{Claim 1:} \textit{The value of the \MDP($\mu$)\ is equal to 1: $v_\mu=1$}.\smallskip

\noindent\textsc{Proof of Claim 1:} Given $n \in \NN$, consider the following strategy $\sigma_n$: play action $B$ at every stage in $E_{n}^{-} = \{k\cdot 2^n-1\colon k\in\NN\}$ and action $T$ at all other odd stages, that is, at stages in $E_{0} \setminus E_{n}^{-}$. Then the decision maker receives reward 1 at each stage of $E_{n}$ as well as at each stage of the set $E_{0} \setminus E_{n}^{-}$. Consequently, the overall payoff is 
\[u_\mu(\sigma_n)\,=\,\mu\big((E_{0} \setminus E_{n}^{-}) \cup E_{n}\big) \,=\, \tfrac{1}{2}\cdot\mu_{0}(E_{0} \setminus E_{n}^{-})  + \tfrac{1}{2}\cdot\mu_{*}(E_{n}).\]
As $E_0\supseteq E_n^{-}$, we have 
\begin{align*}
\mu_{0}(E_{0} \setminus E_{n}^{-}) \,&=\, \mu_{0}(E_0) - \mu_{0}(E_n^{-})\\[0.1cm]
&=\, 1 - \mu_0(\{2^{n}-1,2\cdot 2^{n}-1,3\cdot 2^{n}-1,\ldots\})\\[0.1cm]
&=\, 1 - 2\cdot \phi(\{2^{n}-1,2\cdot 2^{n}-1,3\cdot 2^{n}-1,\ldots\})\\[0.1cm]
&=\, 1 - \tfrac{1}{2^{n-1}},
\end{align*}
where the last equality follows from the fact that $\phi$ is a frequency charge. We have already seen that $\mu_{*}(E_{n}) = 1$. Consequently, the payoff under the strategy $\sigma_n$ is 
\[u_\mu(\sigma_n)\,=\,\tfrac{1}{2}\cdot (1 - \tfrac{1}{2^{n-1}})+\tfrac{1}{2}\cdot 1\,=\,1-\tfrac{1}{2^{n}}.\] Since $n\in\NN$ is arbitrary, the claim follows.\smallskip

\begin{comment}
\noindent\textsc{Claim 2:} \textit{There is no 0-optimal pure strategy}.\smallskip

Consider a pure strategy with payoff 1. Let $S$ be the set of stages where the reward 1 is received. We must have $\mu_{0}(S) = 1$ and $\mu_{*}(S) = 1$. From the first fact we deduce that $\phi(\{k \in \NN: 2k-1 \in S\}) = 1$. On the other hand, if $2k \in S$, then $2k-1 \notin S$. Consequently, $\phi(\{k \in \NN: 2k \in S\}) \leq \phi(\{k \in \NN: 2k-1 \notin S\}) = 0$. It follows that for each $n \in \NN$ we have $\mu_{n}(S) = \phi(\{k \in \NN: 2nk \in S\}) = 0$. But then $\mu_{*}(S) = 0$, a contradiction.\smallskip
\end{comment}

\noindent\textsc{Claim 2:} For any strategy $\sigma$, we have $u_\mu(\sigma)<1$.\smallskip

\noindent\textsc{Proof of Claim 2:} Suppose the opposite: a strategy $\sigma$ gives payoff 1, i.e., $u_\mu(\sigma)=1$. Let $W$ be the set of stages where the expected reward is greater than $\tfrac{1}{2}$, that is $W = \{t \in \NN:\EE_{\sigma}[r_{t}]>\tfrac{1}{2}\}$. We must have $\mu(W) = 1$ and consequently $\mu_{0}(W) = 1$ and $\mu_{*}(W) = 1$. 

Note that, for any $n\in \NN$, we have $(W\cap E_n)-1\subseteq E_0\setminus W$. Indeed, let $k\in (W\cap E_n)-1$. Then $k+1\in E_n$, so $k+1$ is even and $k$ is odd, proving $k\in E_0$. Further, $k+1\in W$. Recall that $r_{k} + r_{k+1} = 1$, hence $\EE_{\sigma}[r_{k}] + \EE_{\sigma}[r_{k+1}] = 1$, and hence it is impossible that the expected reward in both stages $k$ and $k+1$ is greater than $\tfrac{1}{2}$. This proves that $k\notin W$.

We distinguish two cases, and in each of them we derive a contradiction. 

Suppose first that $\mu_n(W)>0$ for some $n\in\NN$. Then, by the definition of $\mu_n$, we have $\phi(W\cap E_n)>0$, and hence the translation invariance of $\phi$ implies that
\[0 < \phi(W\cap E_n) = \phi((W\cap E_n)-1).\]
Consequently
\[\phi(W\cap E_0)\,=\,\phi(E_0)-\phi(E_0\setminus W)\,\leq\,\phi(E_0)-\phi((W\cap E_n)-1)\,<\,\phi(E_0)\,=\,\frac{1}{2},\] so $\mu_0(W)<1$ holds, which is a contradiction.

Now suppose that $\mu_n(W)=0$ for all $n\in\NN$. Because $\mu_*$ is an accumulation point of the sequence $\{\mu_n\}_{n=1}^\infty$, we have $\mu_{*}(W) = 0$, which is again a contradiction. The proof is complete.
\end{proof}

\section{Concluding remarks}\label{sec-concluding}

\noindent\textbf{I. Non-existence of stationary optimal strategies.} For the even--or--odd \MDP, defined in Section \ref{sec-opt}, it is not difficult to find an aggregation charge $\mu'\in\Delta^d$ such that, in the \MDP$(\mu')$, although the decision maker has a pure optimal strategy, none of the stationary strategies are optimal. 

\begin{exl}\label{ex-nostat}\rm 
Consider the even--or--odd \MDP, defined in Section \ref{sec-opt}. Let  
\[Q\,=\,\{4n-3\colon n\in\NN\}\cup\{4n\colon n\in\NN\}\,=\,\{1,4,5,8,9,12,13,\ldots\}.\]
Let $\alpha:\NN\to Q$ be the bijection where $\alpha(n)$ is the $n$th-element of $Q$. Let $\mu'\in\Delta^d$ be a diffuse charge such that $\mu'(\{4n-3\colon n\in\NN\})=\mu'(\{4n\colon n\in\NN\})=\frac{1}{2}$, thus in particular, $\mu'(Q)=1$. 

In the \MDP$(\mu')$, there is a pure strategy that yields payoff 1, and is therefore optimal: choose actions $T$ and $B$ alternately whenever the play is in state 1, so that reward 1 is received in state 1 at every stage in $\{4n-3\colon n\in\NN\}$ and reward 1 is received in state 3 at every stage in $\{4n\colon n\in\NN\}$. 

However, there is no stationary optimal strategy. Indeed, suppose that the stationary strategy chooses action $T$ with probability $q$. Then, the expected reward at each stage in $\{4n-3\colon n\in\NN\}$ is equal to $q$, and the expected reward at each stage in $\{4n\colon n\in\NN\}$ is equal to $1-q$. Hence, the payoff under this strategy is equal to $\frac{1}{2}\cdot q$+$\frac{1}{2}\cdot (1-q)=\frac{1}{2}$, which is strictly less than 1. $\blacklozenge$
\end{exl}

\noindent\textbf{II. The order of integration in the payoff.} The payoff under a strategy $\sigma$, as defined in \Eqref{eqpay}, can be rewritten as an iterated integral. Indeed, for the sake of an easier notation, assume that the action space in each state is the same set $A$. Then,
\begin{equation}\label{eqpay2}
u_\mu(\sigma)\;=\;\int_{t\in\NN}\int_{x\in X}r_t(x)\,\PP_{\sigma}(dx)\,\mu(dt),
\end{equation}
where $X=(S\times A)^\NN$ denotes the set of infinite plays in the \MDP, endowed with the Borel sigma-algebra induced by the product topology, $\PP_{\sigma}$ is the probability measure on $X$ induced by $\sigma$, and $r_t(x)$ is the random variable for the reward at stage $t$ if the infinite play is $x\in X$. We remark that reversing the order of integration in \Eqref{eqpay2}, while interesting, is not always well defined, because the integral $\int_{t\in\NN}r_t(x)\,\mu(dt)$ is not necessarily a measurable function of the infinite play $x$.\medskip

\noindent\textbf{III. Non-diffuse aggregation charges.} 
In Section \ref{mu-average}, we assumed that the aggregation charge is diffuse, which is the most challenging case. When the aggregation charge is a non-diffuse charge $\mu\in\Delta^f\setminus \Delta^d$, we can still define the model in the same way. Since each charge in $\Delta^f$ can be written in a unique way as a convex combination of a countably additive probability measure on $\NN$ and a diffuse charge in $\Delta^d$
(cf., for instance, Theorem 10.2.1 in Rao and Rao [1983]), we have two cases: (C1) $\mu$ is a countably additive probability measure on $\NN$, or (C2)  $\mu$ is a convex combination of a countably additive probability measure on $\NN$ and a diffuse charge such that the weight of the countably additive part is positive. 

In case (C1), it holds that $\sum_{t=1}^\infty\mu(t)=1$, and for every strategy $\sigma$, the payoff is equal to $u_\mu(\sigma)=\sum_{t=1}^\infty \mu(t)\EE_{\sigma}[r_t]$. Therefore, as the action space $A$ is finite, it is known that the decision maker has a pure optimal strategy: in each stage, depending on the current state, play an optimal action in the dynamic programming equation.\footnote{Such a strategy, where the action choice only depends on the current stage and the current state, is called Markov. The existence of a pure Markov optimal strategy follows from many other general statements too, for instance, from Corollary 4.2 in Fudenberg and Levine [1983]. Indeed, for the sake of an easier notation, assume that the action space in each state is the same set $A$. The \MDP\ with finite duration $T$ and evaluation $u^T_\mu(\sigma)=\sum_{t=1}^T \mu(t) \EE_\sigma[r_t]$ admits a pure Markov optimal strategy $\sigma^T$, by backward induction. The set of pure strategies $A^H$ endowed with the product topology, where $H$ is the set of histories, is compact and metrizable (as $H$ is countable), and hence the sequence $\{\sigma^T\}_{T=1}^\infty$ has an accumulation point $\sigma$, which is a pure Markov strategy. Due to a continuity argument, $\sigma$ is optimal.}

In case (C2), the following example shows that an optimal strategy does not always exist. This type of examples are well-known when the decision maker tries to maximize a convex combination of the discounted sum of the rewards and the long-term average reward, just as in the example below.

\begin{exl} \rm Consider the following \MDP. The set of states is $S=\{1,2\}$. In state 1 the decision maker has two actions: $T$ and $B$. If she chooses $T$ then the reward is 1 and the play remains in state~1. If she chooses $B$ then the reward is 0 and the next state is state~2. In state 2 the decision maker has only one action, which gives reward $\frac{3}{2}$ and keeps the play in state~2.\vspace{0.4cm}

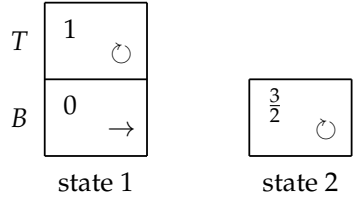
\scalebox{0.85}{
\hspace{5cm}\begin{tikzpicture}[x=0.4cm,y=0.4cm]
\node at (-1,4.5) {$T$};
\node at (-1,1.5) {$B$};
\draw[thick] (0,0) -- (4,0);
\draw[thick] (0,3) -- (4,3);
\draw[thick] (0,6) -- (4,6);
\draw[thick] (0,0) -- (0,6);
\draw[thick] (4,0) -- (4,6);
\node at (1,5) {1};
\node at (1,2) {0};
\node at (3,4) {$\circlearrowright$};
\node at (3,1) {$\rightarrow$};
\node at (2,-1) {state 1};

\draw[thick] (8,0) -- (12,0);
\draw[thick] (8,3) -- (12,3);
\draw[thick] (8,0) -- (8,3);
\draw[thick] (12,0) -- (12,3);
\node at (9,2) {$\frac{3}{2}$};
\node at (11,1) {$\circlearrowright$};
\node at (10,-1) {state 2};
\end{tikzpicture}
}\vspace{0.2cm}

\noindent Finally, define $\mu\in\Delta^f$ as \[\mu\,=\,\frac{1}{2}\mu'+\frac{1}{2}\mu''\] 
where $\mu'$ is the countably additive probability measure on $\NN$ given by the probabilities $(\frac{1}{2},\frac{1}{4},\frac{1}{8},\ldots)$  and $\mu''\in\Delta^d$ is an arbitrary diffuse charge.\footnote{Here $\mu'$ is a discounted evaluation with discount factor $\frac{1}{2}$, and $\mu''$ could be any frequency charge representing an average reward evaluation.} 

Intuitively, the decision maker would like to choose action $B$ at some stage, but preferably as late as possible, so that the incurred loss due to the reward 0 is minimal. This will cause, as we show, that she has no optimal strategy.

We only discuss pure strategies; one can extend the proof to non-pure strategies. Let $T^\infty$ be the pure strategy that always chooses $T$, and for each $n\in\NN$, let $B^n$ denote the pure strategy that chooses $T$ before stage $n$ and chooses $B$ at stage $n$. \smallskip

Consider the payoff $u_\mu(\cdot)$ from \Eqref{eqpay}. We have $u_\mu(T^\infty)=\frac{1}{2}\cdot 1 + \frac{1}{2}\cdot 1=1$ and $u_\mu(B^1)=\frac{1}{2}\cdot (\frac{1}{2}\cdot\frac{3}{2})+\frac{1}{2}\cdot \frac{3}{2}=\frac{9}{8}$. Hence, $T^\infty$ is not optimal. Further, for any $n\in\NN$, the strategies $B^n$ and $B^{n+1}$
induce different rewards only at stage $n$ (reward 0 and 1 respectively) and at stage $n+1$ (reward 3/2 and 0 respectively). Hence, $\mu''(B^n)=\mu''(B^{n+1})$, and we find
\[u_\mu(B^n)-u_\mu(B^{n+1})\,=\,\frac{1}{2}\cdot (\mu'(B^n)-\mu'(B^{n+1}))\,=\,\frac{1}{2}\cdot\Big(\frac{1}{2^{n+1}}\cdot \frac{3}{2}-\frac{1}{2^{n}}\cdot 1\Big)\,<\,0.\] 
Hence, for any $n\in\NN$, the strategy $B^{n+1}$ is strictly better than $B^{n}$. It follows that there is no pure optimal strategy, as claimed. $\blacklozenge$
\end{exl}

\section*{References}

%\noindent Everett, H. [1957]. Recursive games. \emph{Contributions to the Theory of Games}, 3(39), 47-78.\medskip

%\noindent Filar JA, Vrieze OJ [1992]: Weighted reward criteria in competitive Markov decision processes. \emph{Mathematical Methods of Operations Research} 36, 343--358. \medskip

%\noindent Flesch J, Vermeulen D, Zseleva A [2017]: Zero-sum games with charges. \emph{Games and Economic Behavior} 102, 666--686.\medskip

%\noindent Fremlin DH: Measure Theory, Volume 5, Part I.\medskip

\noindent Fudenberg D, Levine D [1983]: Subgame-perfect equilibria of finite- and infinite-horizon games. \emph{Journal of Economic Theory} 31, 251-268.\medskip

%\noindent Hart S [1985]: Nonzero-sum two-person repeated games with incomplete information. \emph{Mathematics of Operations Research} 10, 117--153.\medskip

%\noindent Hrbacek K, Jech T [1999]: \textit{Introduction to Set Theory}. Third edition, revised and Expanded. CRC Press.

\noindent Kadane JB, O'Hagan A [1995]: Using Finitely Additive Probability: Uniform Distributions on the Natural Numbers. \emph{Journal of the American Statistical Association} 90, 626-631.\medskip

%\noindent Mertens J-F, Sorin S, Zamir S [2015]. \emph{Repeated Games} (Vol.~55). Cambridge University Press.\medskip

\noindent Neyman  A [2023]: Additive valuations of streams of payoffs that satisfy the time value of money principle: A characterization and robust optimization. \emph{Theoretical Economics}, 18(1), 303-340.\medskip

%\noindent Nowak AS, Raghavan TES [1991]. Positive stochastic games and a theorem of Ornstein. In: \emph{Stochastic Games And Related Topics: In Honor of Professor LS Shapley} (pp. 127-134). Dordrecht: Springer Netherlands.\medskip

\noindent Puterman ML [1990]: \emph{Markov Decision Processes}. Handbooks in operations research and management science, 2, 331-434.\medskip

\noindent Rao KPSB, Rao MB [1983]: \emph{Theory of Charges}. Academic Press.\medskip

\noindent Schirokauer O, Kadane JB [2007]: Uniform distributions on the natural numbers. \emph{Journal of Theoretical Probability}, 20, 429-441.\medskip

%\noindent Shapley, L. S. [1953]. Stochastic games. \emph{Proceedings of the National Academy of Sciences}, 39(10), 1095-1100.\medskip

%\noindent Sudderth WD [2014]: Random walks and the number theoretic density. \emph{Statistical Methods \& Applications} 23, 477-482.\medskip

%\noindent Venel X, Renault J [2017]: Long-term values in Markov decision processes and repeated games, and a new distance for probability spaces. \emph{Mathematics of Operations Research} 42, 349--376.\medskip

%\noindent Yosida K, Hewitt E. [1952]: Finitely additive measures. \emph{Transactions of the American Mathematical Society}, 72(1), 46-66.\medskip

\end{document}